\documentclass{amsart}
\usepackage{dsfont}
\usepackage{mathrsfs}
\usepackage{color}
\usepackage{CJK}
\usepackage{amssymb}
\newtheorem{theorem}{Theorem}[section]
\newtheorem{lemma}[theorem]{Lemma}
\usepackage[nosort]{cite}
\theoremstyle{definition}
\newtheorem{definition}{Definition}[section]

\newtheorem{corollary}[theorem]{Corollary}

\theoremstyle{remark}

\newtheorem{remark}{Remark}[section]

\numberwithin{equation}{section}

\makeatletter

\newcommand{\Rmnum}[1]{\expandafter\@slowromancap\romannumeral #1@}
\makeatother

\begin{document}

\title{Topological $R$-pressure and topological pressure of free semigroup actions }

\author{Yinan Zheng}
\address{School of Mathematics, South China University of Technology, Guangzhou 510641, P.R. China}
\email{13132591050@163.com}

\author{Qian Xiao}
\address{School of Mathematics, South China University of Technology, Guangzhou 510641, P.R. China}
\email{1309200581qianqian@163.com}




\begin{abstract}
In this paper we introduce the definition of topological $r$-pressure of free semigroup actions on compact metric space and provide some properties of it. Through skew-product transformation into a medium, we can obtain the following two main results. 1. We extend the result that the topological pressure  is the limit of topological $r$-pressure  in\cite{C} to free semigroup actions ($r\to 0$). 2. Let $f_i,$ $i=0, 1, \cdots, m-1$, be homeomorphisms on a  compact metric space. For any continuous function, we verify that the  topological pressure of $f_0, \cdots, f_{m-1}$  equals the topological pressure of $f_0^{-1}, \cdots, f_{m-1}^{-1}.$
\end{abstract}


\maketitle

\section{ Introduction }

The topological entropy was introduced by Adler et al\cite{Adler} as an invariant of topological conjugacy, which describes the complexity of a system. Later, using spanning sets and separated sets, Bowen\cite{Bowen} defined topological entropy for a uniformly continuous map on metric space and proved that for a compact space, they coincide with that defined by Adler et al.
As a natural extension of topological entropy, topological pressure is a rich source of dynamical systems. Ruelle\cite{Ruelle} first introduced the concept of topological pressure of additive potentials for expansive dynamical systems. Walters\cite{Walters,Walters1} then extended this concept to a compact space with the continuous transformation. Since topological pressure plays an important role in dynamical systems, some researchers try to find some generalization of topological pressure applicable to other systems (see, for example,\cite{Barreira,Cao,Chung,Huang, Lin,Ma,Ma1,Thompson,Zhang}).

In 1980, Feldman\cite{Feldman} introduced the concept of $r$-entropy for the interaction between $\mathbb{Z}$ and $\mathbb{R}^{n}$. On this basis, the notion of topological $r$-entropy and measure-theoretic $r$-entropy of a continuous map were introduced in \cite{Ren}.

Let $(X, d)$ be a compact metric space and continuous map $f$ from $X$ into itself. For a finite set $A$, we denote the cardinality of $A$ by Card$A$. Now for $x \in X,~n\geq 0,~\varepsilon>0$,
$0<r<1$ and $\varphi\in C\left( X,\mathbb{R} \right)$, let
\[
B(x,n,\varepsilon, r,f) =\left\{y\in X:\frac{1}{n}Card\{0\leq i \leq n-1:d(f^i x ,f^i y ) < \varepsilon\}> 1-r\right\}.
\]

A subset $F$ is called an $(n,\varepsilon,r,f)$-spanning set of $X$ (with respect to $f$), if for arbitrary $x\in X$, there exists $y\in F$, such that $x\in B(y,n,\varepsilon,r,f)$.

Denote by $r_{n}(f,\varepsilon,r)$ the smallest cardinality of any $(n,\varepsilon,r,f)$-spanning set of $X$. In \cite{Ren}, Ren et al defined
\[
h_{r}{(f,\varepsilon)}=\limsup_{n \to \infty} \frac{1}{n} \log r_{n}(f,\varepsilon,r),
\]
and
\[
h_{r}{(f)}=\lim_{\varepsilon \to 0}h_{r}{(f,\varepsilon)},
\]
$h_{r}{(f)}$ is called the topological $r$-entropy of $f$.

\begin{theorem}(\cite{Ren}, Corollary 2.5)
Suppose that $(X,d)$ is a compact space with the metric $d$, $f:X\rightarrow X$ is a continuous map. Then
\[
\lim_{r\to 0} h_{r}{(f)}=h{(f)},
\]
where $h{(f)}$ denotes the topological entropy of $f$.
\end{theorem}

Moreover, Zhu et al \cite{Zhu} introduced the notion of topological $r$-entropy of free semigroup actions on a compact metric space and proved the following:
\begin{theorem}(\cite{Zhu}, Theorem 1.1)
Let $f_0, \dotsc, f_{m-1}$ be continuous transformations of a compact metric space $(X,d)$ to itself. Then
\[
\lim_{r\to 0}h_r(f_0, f_1, \cdots, f_{m-1})= h(f_0, f_1, \cdots, f_{m-1}),
\]
where $h(f_0, f_1, \cdots, f_{m-1})$ denotes the topological entropy of $f_0, \cdots, f_{m-1}$ and \\
$h_r(f_0, f_1, \cdots, f_{m-1})$  denotes the topological $r$-entropy of $f_0, \cdots, f_{m-1}$(\cite{Zhu},definition 3.1).
\end{theorem}

In \cite{C}, Chen further introduced the concept of topological $r$-pressure in compact metric space and gave relevant properties.

For $\varphi\in C\left( X,\mathbb{R} \right)$, put
\[
Q_{n}\left(f,  \varphi, \varepsilon,r\right)=\inf \left\{\sum_{x \in F} e^{\left(S_{n} \varphi\right)(x)} : F \text { is } a\left(n, \varepsilon, r,f\right)\text {-spanning set of $X$}\right\}.
\]

In\cite{C}, Chen defined
\[
P_r(f,\varphi) =\lim_{\varepsilon \to 0} \limsup_{n \to \infty} \frac{1}{n} \log Q_{n}\left(f,  \varphi, \varepsilon,r\right),
\]
$P_r(f,\varphi)$ is called the topological $r$-pressure of $f$.

\begin{theorem}\label{222}(\cite{C}, Corollary 3.2.2)
Let $f\colon X\to X$ be a continuous map of a compact metric space $(X,d)$, and $\varphi\in C\left( X, \mathbb{R} \right)$. Then
\[
\lim_{r \to 0} P_r(f,\varphi) = P(f,\varphi),
\]
where $P(f,\varphi)$ denotes the topological pressure of $f$.
\end{theorem}

Based on the above results, we introduce the notion of topological $r$-pressure for a free semigroup action and give some properties of the notion in this paper. The main results of this paper are the following two theorems.

\begin{theorem}\label{11111}
Let $f_0, \dotsc, f_{m-1}$ be continuous transformations of a compact metric space $(X,d)$ to itself, and $\varphi\in C\left( X,\mathbb{R} \right)$. Then
\begin{equation}\label{them:topology1}
\lim_{r\to 0}P_r(f_0, f_1, \cdots, f_{m-1},\varphi)= P(f_0, f_1, \cdots, f_{m-1},\varphi),
\end{equation}
where $P(f_0, f_1, \cdots, f_{m-1},\varphi)$ denotes the topological pressure of $f_0, \cdots, f_{m-1}$ and $P_r(f_0, f_1, \cdots, f_{m-1},\varphi)$  denotes the topological $r$-pressure of $f_0, \cdots, f_{m-1}$.
\end{theorem}

\begin{theorem}\label{0000}
Let $f_0, \dotsc, f_{m-1}$ be homeomorphisms of a compact metric space $(X,d)$ to itself, and  $\varphi\in C\left( X,\mathbb{R} \right)$. Then
\begin{equation}
P(f_0, \cdots, f_{m-1},\varphi) = P(f_0^{-1}, \cdots, f_{m-1}^{-1},\varphi).
\end{equation}
\end{theorem}

This paper is organized as follows. In section 2, we give some preliminaries. In section 3, we give the definition of topological $r$-pressure of free semigroup actions and give some properties of them. In section 4, we give the proof of Theorem 1.4. In section 5, we give the proof of Theorem 1.5.

\section{Preliminaries}

\subsection{Topological pressure of free semigroup actions}\

Let $F_m^+$ be the set of all finite words of symbols $0,1,\cdots,m-1$. For every $w \in F_m^+$, $|w|$ denotes the length of $w$, i.e., the number of symbols in $w$. If $w, w' \in F_m^+$, define $ww'$ be the word obtained by writing $w'$ to the right of $w$. With respect to this law of composition, $F_m^+$ is a free semigroup with $m$ generators. We write $w \leq w'$ if there exists a word $w''$ such that $w' = w''w$.

Denote  the set of all two-side infinite sequences of symbols $0,1,\ldots,m-1$ by $\Sigma_m$, that is, \[ \Sigma_m=\{\omega=(\cdots, \omega_{-1}, \overset{*}{\omega_0}, \omega_1, \cdots) : \omega_i=0,1,\ldots,m-1~for~all~integer~i\}.\]

A metric on $\Sigma_m$ is introduced by setting \[d'(\omega,\omega')=1/2^k,~where~k=\inf\{|n|:\omega_n\neq \omega_n' \}.\]

Obviously, $\Sigma_m$ is compact with respect to this metric. The Bernoulli shift $\sigma: \Sigma_m \rightarrow \Sigma_m$ is a homeomorphism of $\Sigma_m$ given by the formula: \[(\sigma \omega)_i=\omega_{i+1}.\]

Assume that $\omega \in \Sigma_m,w\in F_m^{+}$, $a,b$ are integers, and $a\leq b$. We write $\omega|_{[a,b]}=w$ if $w=\omega_a\omega_{a+1}\ldots\omega_{b-1}\omega_b$.

Suppose that a free semigroup with $m$ generators acts on $X$,  denote the maps corresponding to the generators by $f_0, \dotsc, f_{m-1}$, we assume that these maps are continuous. If $w\in F_m^+$, $w = w_k w_{k-1} \cdots w_1$ where $w_i \in \{0,1,\cdots, m-1\}$ for all $i=1,2,\cdots, k$, denote $f_w = f_{w_k}f_{w_{k-1}}\cdots f_{w_1}$. Obviously, $f_{ww'}= f_w f_{w'}$ for any $w,w' \in F_m^+$.

For every $w\in F_m^+$, define a metric $d_w$ on $X$ by
\[
d_w(x_1, x_2) = \max_{w' \leq w} d(f_{w'}(x_1), f_{w'}(x_2)), \forall x_1, x_2 \in X.
\]

Obviously, if $w \leq w'$, then $d_w(x_1,x_2) \leq d_{w'}(x_1,x_2)$  for all $x_1,~x_2\in X$.

For every $w\in F_m^+,~{w'}\leq w,~\varphi\in C\left( X,\mathbb{R} \right)$ and we denote $\sum _{w'\leq w}\varphi \left( f_{w'}x \right)$ by $\left( S_{w}\varphi \right) \left( x\right)$.

Let $\varepsilon>0,$ a subset $E$ of $X$ is said to be a $(w, \varepsilon, f_0, \ldots, f_{m-1})$-separated set of $X$ for any $x,y \in E$, $x \neq y$, implies $d_w(x,y) > \varepsilon$.

In \cite{Lin}, Lin et al defined
\begin{align*}
&Q_{w}^{s}\left(f_{0}, \ldots, f_{\mathrm{m}-1}, \varphi, \varepsilon\right)\\
=&\sup \left\{\sum_{x \in E} e^{\left(S_{w} \varphi\right)(x)} : E \text { is } a\left(w, \varepsilon, f_{0}, \ldots, f_{\mathrm{m}-1}\right)\text{-separated set of $X$}\right\},
\end{align*}
\[
Q_{n}^{s}( f_0, \ldots, f_{m-1},\varphi,\varepsilon) = \frac{1}{m^n} \sum_{|w|=n}Q_{w}^{s}( f_0, \ldots, f_{m-1},\varphi,\varepsilon).
\]

Let $\varepsilon>0,$ a subset $F$ of $X$ is called a $(w, \varepsilon, f_0, \ldots, f_{m-1})$-spanning set of $X$ if for every $x\in X$ there exists $y \in F$ such that $d_w(x,y) \leq \varepsilon$.

Let
\begin{align*}
&Q_{w}\left(f_{0}, \ldots, f_{\mathrm{m}-1}, \varphi, \varepsilon\right)\\
=&\inf \left\{\sum_{x \in F} e^{\left(S_{w} \varphi\right)(x)} :F \text { is } a\left(w, \varepsilon, f_{0}, \ldots, f_{\mathrm{m}-1}\right)\text{-spanning set of $X$}\right\}
\end{align*}
\[
Q_{n}( f_0, \ldots, f_{m-1},\varphi,\varepsilon) = \frac{1}{m^n} \sum_{|w|=n}Q_{w}( f_0, \ldots, f_{m-1},\varphi,\varepsilon),
\]
and defined the topological pressure of a  free semigroup action by the formula
\begin{align*}
P(f_0, \ldots, f_{m-1},\varphi) &= \lim_{\varepsilon \to 0} \limsup_{n \to \infty} \frac{1}{n} \log Q_{n}( f_0, \ldots, f_{m-1},\varphi,\varepsilon)\\
 &= \lim_{\varepsilon \to 0} \limsup_{n \to \infty} \frac{1}{n} \log Q_{n}^{s}( f_0, \ldots, f_{m-1},\varphi,\varepsilon).
\end{align*}

\begin{remark}
When $m =1$, it coincides with the topological pressure defined in \cite{Walters,Walters1}. Sometimes, in order to emphasize metric $d$, the topological pressure is always denoted as $P_{d}(f_{0},\ldots,f_{m-1},\varphi)$.
\end{remark}

\subsection{Skew-product transformation }\

Let $(X,d)$ be a compact metric space and suppose $f_0, \ldots, f_{m-1}$ are continuous mappings of $X$ to itself.

To this action, we assign the following a skew-product transformation. Its base is $\Sigma_m$, its fiber is $X$, and the maps $F: \Sigma_m \times X \to \Sigma_m \times X$ and  $g:\Sigma_m \times X \to \mathbb{R}$ are defined by the formula
\[
F(\omega, x) = (\sigma \omega, f_{\omega_0}(x)),
\]
\[
g(\omega, x)=\varphi(x),
\]
where $\omega = (\cdots, \omega_{-1}, \overset{*}{\omega_0}, \omega_1, \cdots)$, $\sigma $ is the Bernoulli shift and $\varphi\in C\left( X,\mathbb{R} \right)$.
Here $f_{\omega_0}$ stands for $f_0$ if $\omega_0=0$, and for $f_1$ if $\omega_0=1$, and so on. For $w=i_1\cdots i_k \in F_m^{+}$, denote $\overline{w}=i_k\cdots i_1$. Let $\omega= (\cdots, \omega_{-1}, \overset{*}{\omega_0}, \omega_1, \cdots)\in \Sigma_m$, then
\begin{align*}
F^n(\omega,x)&=(\sigma^n\omega,f_{\omega_{n-1}}f_{\omega_{n-2}}\cdots f_{\omega_0}(x))\\&=(\sigma ^n\omega,f_{\overline{\omega|_{[0,n-1]}}}(x)).
\end{align*}

Let $P(F,g)$ denote the topological pressure of $F$ with respect to $g$ for the dynamical system $(\Sigma_m \times X,F)$.

In \cite{Lin}, Lin et al proved the following  theorem.

\begin{theorem}\label{them:topology2}(\cite{Lin}, Theorem 1.1 )
The topological pressure of the skew product transformation $F$ with respect to $g$, satisfies
\[
P_{D}(F,g) =\log m +P_{d}(f_0, \cdots, f_{m-1},\varphi),
\]
where the metric $D$ on $\Sigma_m \times X$ is defined as
\[
D((\omega,x), (\omega', x')) = \max \left\{d'(\omega, \omega'),d(x, x')\right\}
\]
and the metric $d'$ on $\Sigma_m$ is introduced by setting $d'(\omega,\omega')=1/2^k$, and $k=\inf\{|n|:\omega_n\neq\omega_n^{'} \}$.
\end{theorem}

\section{topological $R$-pressure of free semigroup actions}
In this section, we introduce the definition of  topological $r$-pressure of free semigroup actions and give some properties of the concept.

Let $(X, d)$ be a compact  metric space and $f_0, f_1, \cdots, f_{m-1}$ are continuous maps from $X$ into itself. For $x \in X$, $ w\in F_m^+$, $\varepsilon > 0$ and $0<r<1$, let
\begin{align*}
&B(x,w,\varepsilon, r,f_0 \ldots f_{m-1})\\
=&\left\{y\in X:\frac{1}{|w|}Card\{w':d(f_{w'} x ,f_{w'} y ) < \varepsilon , and~w' \leq w\}> 1-r\right\}.
\end{align*}


A subset $F$ of $X$ is called a $(w,\varepsilon, r,f_0 \ldots f_{m-1})$-spanning set of $X$, if for every $x\in X$, there exists $y \in F$ such that $x\in B(y,w,\varepsilon, r,f_0 \ldots f_{m-1}).$

Let
\begin{align*}
&Q_{w}\left(f_{0}, \ldots, f_{\mathrm{m}-1}, \varphi, \varepsilon,r\right)\\
=&\inf \left\{\sum_{x \in F} e^{\left(S_{w} \varphi\right)(x)} : F \text { is a} \left(w, \varepsilon,r, f_{0}, \ldots, f_{\mathrm{m}-1}\right)\text {-spanning set of}~X \right\},
\end{align*}
\[
Q_{n}( f_0, \ldots, f_{m-1},\varphi,\varepsilon,r) = \frac{1}{m^n} \sum_{|w|=n}Q_{w}( f_0, \ldots, f_{m-1},\varphi,\varepsilon,r).
\]

\begin{remark}
Let $r_{w}(f_0, \ldots, f_{m-1},\varepsilon,r)$ denote the smallest cardinality of any\\
$(w,\varepsilon,r,f_0, \ldots, f_{m-1})$-spanning set of $X$. Then let
\[
r_{n}( f_0, \ldots, f_{m-1},\varepsilon,r) = \frac{1}{m^n} \sum_{|w|=n}r_{w}( f_0, \ldots, f_{m-1},\varepsilon,r).
\]

$\left( 1\right)$ If $\varepsilon_{1}< \varepsilon_{2}$ then $Q_{w}( f_0, \ldots, f_{m-1},\varphi,\varepsilon_{1},r)\geq Q_{w}( f_0, \ldots, f_{m-1},\varphi,\varepsilon_{2},r)$. Hence,
\[
Q_{n}( f_0, \ldots, f_{m-1},\varphi,\varepsilon_{1},r)\geq Q_{n}( f_0, \ldots, f_{m-1},\varphi,\varepsilon_{2},r).
\]

$\left( 2\right)$ $0<Q_{w}( f_0, \ldots, f_{m-1},\varphi,\varepsilon,r)\leq\parallel e^{S_{w} \varphi}\parallel r_{w}(f_0, \ldots, f_{m-1},\varepsilon,r)$. Hence,
\[
0<Q_{n}( f_0, \ldots, f_{m-1},\varphi,\varepsilon,r)\leq e^{n\parallel\varphi\parallel} r_{n}(f_0, \ldots, f_{m-1},\varepsilon,r).
\]

$\left( 3\right)$ $Q_{w}( f_0, \ldots, f_{m-1},0,\varepsilon,r)=r_{w}(f_0, \ldots, f_{m-1},\varepsilon,r)$. Hence,
\[
Q_{n}( f_0, \ldots, f_{m-1},0,\varepsilon,r)=r_{n}(f_0, \ldots, f_{m-1},\varepsilon,r).
\]

\end{remark}

\begin{definition}
For $0<r<1,$ $\varphi\in C\left( X, \mathbb{R} \right)$, the topological $r$-pressure of free semigroup actions is  defined by the formula
\[
P_r(f_0, \ldots, f_{m-1},\varphi) =\lim_{\varepsilon \to 0} \limsup_{n \to \infty} \frac{1}{n} \log Q_{n}( f_0, \ldots, f_{m-1},\varphi,\varepsilon,r).
\]
\end{definition}
We sometimes  write $P_{r,d}(f_0, \ldots, f_{m-1},\varphi)$ to emphasise the dependence on  $d$.

\begin{remark}\label{1234}
 It's obvious  that $P(f_0, \ldots, f_{m-1},\varphi)\geq P_r(f_0, \ldots, f_{m-1},\varphi).$ If $m =1$, it coincides with the topological $r$-pressure defined in \cite{C} and if $\varphi =0$, it coincides with the topological $r$-entropy of free semigroup actions defined in \cite{Zhu}.
\end{remark}

Now we make a brief introduction to the  separated set.

A subset $E$ of $X$ is called a $(w, \varepsilon, r,f_0,\ldots, f_{m-1})$-separated set of $X$ for any $x,y \in E$, $x \neq y$, implies
\[
\frac{1}{|w|}Card\{w':d(f_{w'} x ,f_{w'} y ) \geq \varepsilon, and~w' \leq w\}> r.
\]

Put
\begin{align*}
&Q_{w}^s\left(f_{0}, \ldots, f_{\mathrm{m}-1}, \varphi, \varepsilon,r\right)\\
=&\sup \left\{\sum_{x \in E} e^{\left(S_{w} \varphi\right)(x)} : E \text { is a} \left(w, \varepsilon,r, f_{0}, \ldots, f_{\mathrm{m}-1}\right)\text {-separated set of $X$} \right\}
\end{align*}
\[
Q_{n}^s( f_0, \ldots, f_{m-1},\varphi,\varepsilon,r) = \frac{1}{m^n} \sum_{|w|=n}Q_{w}^{s}( f_0, \ldots, f_{m-1},\varphi,\varepsilon,r).
\]

\begin{remark}
If $\varepsilon_{1}<\varepsilon_{2}$ then $Q_{w}^s( f_0, \ldots, f_{m-1},\varphi,\varepsilon_{1},r)\geq Q_{w}^s( f_0, \ldots, f_{m-1},\varphi,\varepsilon_{2},r)$. Hence,
\[
Q_{n}^s( f_0, \ldots, f_{m-1},\varphi,\varepsilon_{1},r)\geq Q_{n}^s( f_0, \ldots, f_{m-1},\varphi,\varepsilon_{2},r).
\]
\end{remark}

Put
\[
P^s_r(f_0, \ldots, f_{m-1},\varphi) = \lim_{\varepsilon \to 0}\limsup_{n \to \infty} \frac{1}{n} \log Q_{n}^s( f_0, \ldots, f_{m-1},\varphi,\varepsilon,r).
\]

We sometimes  write $P^s_{r,d}(f_0, \ldots, f_{m-1},\varphi)$ to emphasise the dependence on  $d$.

\begin{lemma}\label{lem:ineq}
For any $n\geq1$, $\varepsilon > 0$ and $0<r<1$, we have

$\left( 1\right)~Q_{n}(f_0, \ldots, f_{m-1},\varphi,\varepsilon,r,)\leq Q_{n}^s(f_0, \ldots, f_{m-1},\varphi,\varepsilon,r).$

$\left( 2\right)$ Let $M=\max\limits_{x\in X}{|\varphi(x)|}$, and if $\delta=\sup\{|\varphi(x)-\varphi(y)|:d(x,y)\leq\varepsilon\}$, then
\[
Q_{n}^s(f_0, \ldots, f_{m-1},\varphi,2\varepsilon,2r)\leq e^{n\delta+2nrM}Q_{n}(f_0, \ldots, f_{m-1},\varphi,\varepsilon,r).
\]
\end{lemma}

\begin{proof}

$\left( 1\right)$ Let $E$ be a $(w,\varepsilon,r,f_0, \ldots, f_{m-1})$-separated set of $X$ of maximum cardinality, then $E$ is a $(w,\varepsilon,r,f_0, \ldots, f_{m-1})$-spanning set of $X$.

Suppose $E$ is not a $(w,\varepsilon,r,f_0, \ldots, f_{m-1})$-spanning set, then there exists at least one point $x\in X$ such that for all $y\in E$, we have
\[
\frac{1}{|w|}Card\{w':d(f_{w'} x ,f_{w'} y ) <\varepsilon, and~w' \leq w\}<1-r,
\]
that is
\[
\frac{1}{|w|}Card\{w':d(f_{w'} x ,f_{w'} y )\geq \varepsilon, and~w' \leq w\}> r,
\]
which contradicts  the separated set of maximum cardinality.

So we obtain that $(w,\varepsilon,r,f_0, \ldots, f_{m-1})$-separated set which cannot be enlarge to a $(w,\varepsilon,r,f_0, \ldots, f_{m-1})$-separated set must be a $(w,\varepsilon,r,f_0,\ldots, f_{m-1})$-spanning set of $X$. Therefore we have
\[
Q_{w}(f_0, \ldots, f_{m-1},\varphi,\varepsilon,r)\leq Q_{w}^s(f_0,\ldots, f_{m-1},\varphi,\varepsilon,r).
\]

Moreover,
\[
Q_{n}(f_0, \ldots, f_{m-1},\varphi,\varepsilon,r)\leq Q_{n}^s(f_0,\ldots, f_{m-1},\varphi,\varepsilon,r).
\]

$\left( 2\right)$ To show the other lemma, suppose $E$ is a $(w,2\varepsilon,2r,f_0, \ldots, f_{m-1})$-separated set of $X$ and
 $F$ is a $(w,\varepsilon,r,f_0, \ldots, f_{m-1})$-spanning set of $X$.
Define $\phi$ : $E\rightarrow F$  by choosing, for each $x\in E$, some point $\phi(x)\in F$ with $x\in B(\phi(x),w,\varepsilon, r,f_0, \ldots,f_{m-1})$.
Then $\phi(x)$ is injective (details are given in \cite{Zhu}).
Let $z\in E$,~such that
\[
S_{w}\varphi(\phi(z))-S_{w}\varphi(z)=\min_{x\in E}\{S_{w}\varphi(\phi(x))-S_{w}\varphi(x)\}.
\]

Let $A=\{w': d(f_{w'}z, f_{w'}\phi(z))<\varepsilon, w'\leq w\}$, then $Card A >|w|(1-r)$.
Denote $Card A =|w|(1-r)\cdot k$, $1<k\leq\frac{1}{1-r}$.
Then
\begin{align*}
\sum _{y \in F} e^{S_{w} \varphi(y)}\geq \sum _{y \in \phi E} e^{S_{w} \varphi(y)}&=\sum _{x \in E} e^{S_{w} \varphi(\phi(x))-S_{w}\varphi(x)}e^{S_{w} \varphi(x)}\\
&\geq \min _{x \in E} \{e^{{S_{w} \varphi(\phi(x))-S_{w}\varphi(x)}}\}\sum _{x \in E} e^{S_{w} \varphi(x)}\\
&=e^{{S_{w} \varphi(\phi(z))-S_{w}\varphi(z)}}\sum _{x \in E} e^{S_{w} \varphi(z)}\\
&\geq e^{-|w|(1-r)k\delta-|w|(1-(1-r)\cdot k)\cdot 2M}\sum _{x \in E} e^{S_{w} \varphi(x)}\\
&\geq e^{-|w|\delta-|w|r\cdot 2M}\sum _{x \in E} e^{S_{w} \varphi(x)}.
\end{align*}

Therefore,
\[
Q_{w}^s(f_0, \ldots, f_{m-1},\varphi,2\varepsilon,2r)\leq e^{|w|\delta+2|w|rM}Q_{w}(f_0, \ldots, f_{m-1},\varphi,\varepsilon,r).
\]

Hence,
\[
Q_{n}^s(f_0, \ldots, f_{m-1},\varphi,2\varepsilon,2r)\leq e^{n\delta+2nrM}Q_{n}(f_0, \ldots, f_{m-1},\varphi,\varepsilon,r).
\]
\end{proof}

\begin{theorem}\label{11}
Let $(X, d)$ be a compact metric space and $f_0, f_1, \ldots, f_{m-1}$ are continuous maps  from $X$ into itself, and $\varphi\in C\left( X,\mathbb{R} \right)$, then
\[
P^s_{2r}(f_0, f_1, \ldots, f_{m-1},\varphi)-2Mr\leq P_{r}(f_0, f_1, \ldots, f_{m-1},\varphi)\leq P^s_r(f_0, f_1, \ldots, f_{m-1},\varphi).
\]
\end{theorem}

\begin{proof}
It is clear that  Lemma \ref{lem:ineq} implies Theorem  \ref{11}.
\end{proof}


\begin{remark}\label{separe top=top}
If $m=1$ we have $\lim\limits_{r \to 0}{P^s_r(f,\varphi)}=\lim\limits_{r \to 0}{P_r(f,\varphi)}=P(f,\varphi)$ by the Theorem \ref{222} and Theorem \ref{11}.
\end{remark}

Now we study the properties of $P_{r}(f_{0},f_{1},\cdots,f_{m-1},\cdot)$ and $P_{r}^{s}(f_{0},f_{1},\cdots,f_{m-1},\cdot)$.
\begin{theorem}
Let $f_{i}:X\rightarrow X$ be continuous transformations of a compact metric space $(X,d),i=0,1,\cdots,m-1$. If $\varphi,\psi \in C(X,\mathbb{R}),~\varepsilon>0,~0<r<1$ and $c \in \mathbb{R}$, then the following are true. \\
$\left( 1\right)$ $P_{r}(f_{0},f_{1},\cdots,f_{m-1},0)=h_{r}(f_{0},f_{1},\cdots,f_{m-1}).$\\
$\left( 2\right)$ $\varphi \leq \psi$ implies  $P_{r}(f_{0},f_{1},\cdots,f_{m-1},\varphi)\leq P_{r}(f_{0},f_{1},\cdots,f_{m-1},\psi)$. In particular, $h_{r}(f_{0},f_{1},\cdots,f_{m-1})+\inf \varphi \leq P_{r}(f_{0},f_{1},\cdots,f_{m-1},\varphi) \leq h_{r}(f_{0},f_{1},\cdots,f_{m-1})+\sup \varphi.$\\
$\left( 3\right)$  $P_{r}(f_{0},f_{1},\cdots,f_{m-1},\cdot)$ is either finite-valued or constantly $\infty$.\\
$\left( 4\right)$ $|P_{r}^{s}(f_{0},f_{1},\cdots,f_{m-1},\varphi,\varepsilon)-P_{r}^{s}(f_{0},f_{1},\cdots,f_{m-1},\psi,\varepsilon)|\leq\|\varphi-\psi\|$ and so if $P_{r}^{s}(f_{0},f_{1},\cdots,f_{m-1},\cdot) <\infty$, then $|P_{r}^{s}(f_{0},f_{1},\cdots,f_{m-1},\varphi)-P_{r}^{s}(f_{0},f_{1},\cdots,f_{m-1},\psi)| \\\leq \|\varphi-\psi\|.$ \\
$\left( 5\right)$ $P_{r}^{s}(f_{0},f_{1},\cdots,f_{m-1},\cdot,\varepsilon)$ is convex, and so if $P_{r}^{s}(f_{0},f_{1},\cdots,f_{m-1},\cdot) <\infty$, then $P_{r}^{s}(f_{0},f_{1},\cdots,f_{m-1},\cdot)$ is convex.\\
$\left( 6\right)$ $P_{r}(f_{0},f_{1},\cdots,f_{m-1},\varphi+c)=P_{r}(f_{0},f_{1},\cdots,f_{m-1},\varphi)+c$.\\
$\left( 7\right)$ $P_{r}^{s}(f_{0},f_{1},\cdots,f_{m-1},\varphi+\psi)\leq P_{r}^{s}(f_{0},f_{1},\cdots,f_{m-1},\varphi)+P_{r}^{s}(f_{0},f_{1},\cdots,f_{m-1},\psi)+\log m$.\\
$\left( 8\right)$ $P_{r}^{s}(f_{0},f_{1},\cdots,f_{m-1},c \varphi)\leq c P_{r}^{s}(f_{0},f_{1},\cdots,f_{m-1},\varphi)+(c-1)\log m$ if $c \geq 1$ and $P_{r}^{s}(f_{0},f_{1},\cdots,f_{m-1},c \varphi)\geq c P_{r}^{s}(f_{0},f_{1},\cdots,f_{m-1},\varphi)+(c-1)\log m$ if $c \leq 1$ .\\
$\left( 9\right)$ $-2 \log m - P_{r}^{s}(f_{0},f_{1},\cdots,f_{m-1},|\varphi|) \leq P_{r}^{s}(f_{0},f_{1},\cdots,f_{m-1},\varphi) \leq P_{r}^{s}(f_{0},f_{1},\cdots,\\f_{m-1},|\varphi|).$\\
\end{theorem}
\begin{proof}
The proof of the theorem $(1)-(8)$ is analogous to the Lin et al in \cite{Lin}, so we omit the proof. Then we prove the $(9)$.

For $w\in F_{m}^{+}$, let $E$ is a $(w,\varepsilon,r,(f_{0},\ldots,f_{m-1})$-separated subset of $X$.
Since $-\mid\varphi\mid\leq\varphi\leq\mid\varphi\mid$, we have
\[
\sum_{x \in E} e^{\left(S_{w} (-\mid\varphi\mid)\right)(x)}\leq\sum_{x \in E} e^{\left(S_{w}\varphi\right)(x)}\leq\sum_{x \in E} e^{\left(S_{w} \mid\varphi\mid\right)(x)}.
\]

Therefore,
\[
Q_{w}^s(f_0, \ldots, f_{m-1},-\mid\varphi\mid,\varepsilon,r)\leq Q_{w}^s(f_0, \ldots, f_{m-1},\varphi,\varepsilon,r)\leq Q_{w}^s(f_0, \ldots, f_{m-1},\mid\varphi\mid,\varepsilon,r).
\]

Hence,
\[
Q_{n}^s(f_0, \ldots, f_{m-1},-\mid\varphi\mid,\varepsilon,r)\leq Q_{n}^s(f_0, \ldots, f_{m-1},\varphi,\varepsilon,r)\leq Q_{n}^s(f_0, \ldots, f_{m-1},\mid\varphi\mid,\varepsilon,r).
\]

Then
\[
P_{r}^s(f_0, \ldots, f_{m-1},-\mid\varphi\mid)\leq P_{r}^s(f_0, \ldots, f_{m-1},\varphi)\leq P_{r}^s(f_0, \ldots, f_{m-1},\mid\varphi\mid).
\]

From $\left( 8\right)$ we have
\[
-P_{r}^s(f_0, \ldots, f_{m-1},\mid\varphi\mid)-2\log m\leq P_{r}^s(f_0, \ldots, f_{m-1},-\mid\varphi\mid).
\]

Thus,
\[
-P_{r}^s(f_0, \ldots, f_{m-1},\mid\varphi\mid)-2\log m\leq P_{r}^s(f_0, \ldots, f_{m-1},\varphi)\leq P_{r}^s(f_0, \ldots, f_{m-1},\mid\varphi\mid).
\]
\end{proof}

Now we investigate how $P_{r}(f_{0},f_{1},\cdots,f_{m-1},\cdot)$ depends on $f_{0},f_{1},\cdots,f_{m-1}$.
\begin{theorem}
If $(X_{1},d_{1})$, $(X_{2},d_{2})$ are compact metric spaces. Suppose continuous maps $f_0, \ldots, f_{m-1}$ on $X_{1}$ and continuous maps $g_0, \ldots, g_{m-1}$ on $X_{2}$, if $\pi:~X_{1}~\rightarrow~X_{2}$ is a continuous surjective such that $\pi\circ f_i = g_i\circ \pi$ for any $0\leq i\leq m-1$. Then
\[
P_{r}\left(g_{0},\ldots, g_{m-1}, \varphi\right) \leq P_{r}\left(f_{0},\ldots, f_{m-1}, \varphi \circ \pi\right),~\forall \varphi\in C(X_{2}, \mathbb{R}).
\]
If $\pi$ is homeomorphism, then
\[
P_{r}\left(g_{0},\ldots, g_{m-1}, \varphi\right)=P_{r}\left(f_{0},\ldots, f_{m-1}, \varphi \circ \pi\right).
\]
\end{theorem}

\begin{proof}
Let $\varepsilon>0, and$ pick $\delta>0$ so that $d_{1}(x, y)<\delta$ implies $d_{2}(\pi(x), \pi(y))<\varepsilon$. Given $w=i_{1}i_{2}\cdots i_{n} \in F_{m}^{+}$ and $0<r<1$. If $F$ is a $\left(w, \delta, r, f_{0}, \ldots, f_{m-1}\right)$-spanning set of $X_{1},$ then $\pi \left(F\right)={\{\pi\left(x\right):x\in F\}}$ is a
$\left(w, \varepsilon, r, g_{0}, \ldots, g_{m-1}\right)$-spanning set of $X_2$. Then
\begin{align*}
&\sum_{x \in F} e^{\varphi\circ\pi(x)+\varphi\circ\pi(f_{i_{1}}(x))+\cdots+\varphi\circ\pi(f_{{i_{n-1}}{i_{n-2}}\ldots{i_{1}}}(x))}\\
&=\sum_{x \in F} e^{\varphi\circ\pi(x)+\varphi\circ{g_{i_{1}}}\circ\pi(x)+\cdots+\varphi\circ{g_{{i_{n-1}}{i_{n-2}}\ldots{i_{1}}}}\circ\pi(x)}\\
&\geq\sum_{y \in \pi(F)} e^{\varphi(y)+\varphi\circ{g_{i_{1}}(y)}+\cdots+\varphi\circ{g_{{i_{n-1}}{i_{n-2}}\ldots{i_{1}}}(y)}}\\
&\geq Q_{w}\left(g_{0}, \ldots, g_{m-1}, \varphi ,\varepsilon,r\right).
\end{align*}

Thus, we have
\[
 Q_{w}\left(f_{0}, \ldots, f_{m-1}, \varphi \circ \pi, \delta,r\right)  \geq Q_{w}\left(g_{0}, \ldots, g_{m-1}, \varphi, \varepsilon, r\right).
\]

Then
\[
Q_{n}\left(f_{0}, \ldots, f_{m-1}, \varphi \circ \pi, \delta,r\right)\geq Q_{n}\left(g_{0}, \ldots, g_{m-1}, \varphi, \varepsilon, r\right).
\]

Whence, taking logarithms and limits ($\delta\rightarrow0$, have $\varepsilon\rightarrow0$), we obtain that
\[
P_{r}\left(f_{0}, \ldots, f_{n-1}, \varphi \circ \pi\right) \geq P_{r}\left(g_{0}, \ldots, g_{n-1}, \varphi\right).
\]

If $\pi$ is a homeomorphism then we can apply the above with $f_{i},~g_{i},~\pi,~\varphi$ replaced by
$g_{i}, f_{i}, \pi^{-1}, \varphi \circ \pi$ to give $P_{r}\left(g_{0}, \ldots, g_{n-1}, \varphi\right) \geq P_{r}\left(f_{0}, \ldots, f_{n-1}, \varphi \circ \pi\right)$.

\end{proof}

\begin{theorem}
Let $(X_i, d_i)$ be a compact  metric space with the metric $d_i$ and let $\mathcal{F}^{(i)}$ be the set of finite continuous transformations from $X_{i}$ into itself $(i=1,2)$, where $\mathcal{F}^{(1)} = \{f_0^{(1)}, \ldots, f_{m-1}^{(1)}\}$  and $\mathcal{F}^{(2)} = \{f_0^{(2)}, \ldots, f_{k-1}^{(2)}\}$.
If $\varphi_{i}\in C\left( X_{i},\mathbb{R} \right)$ and $\mathcal{F}^{(1)}$  satisfies $P^s_{{r,d}_1}(\mathcal{F}^{(1)},\varphi_{1})=\lim \limits_{\varepsilon\to 0}\liminf\limits_{n\to \infty}\frac{1}{n} \log Q_{n}^{s}\left(\mathcal{F}^{(1)}, \varphi_{1}, \varepsilon, r\right)$ or  $\mathcal{F}^{(2)}$  satisfies
$P^s_{{r,d}_2}(\mathcal{F}^{(2)},\varphi_{2})=\lim \limits_{\varepsilon\to 0}\liminf\limits_{n\to \infty}\frac{1}{n} \log Q_{n}^{s}\left(\mathcal{F}^{(2)}, \varphi_{2}, \varepsilon, r\right)$, then
\[
P^s_{r,d}(\mathcal{F}^{(1)} \times \mathcal{F}^{(2)},\varphi_{1}\times\varphi_{2} ) \geq P^s_{{r,d}_1}(\mathcal{F}^{(1)},\varphi_{1}) +P^s_{{r,d}_2}(\mathcal{F}^{(2)},\varphi_{2}),
\]
where $\mathcal{F}^{(1)} \times \mathcal{F}^{(2)}$ denotes the semigroup acting on the compact space $X_1\times X_2$ generated by $\mathcal{F}^{(1)} \times \mathcal{F}^{(2)} = \{f\times g : f \in \mathcal{F}^{(1)}, g \in \mathcal{F}^{(2)}\}:= \{(f\times g)_0, \ldots, (f \times g)_{mk -1}\}$ where  $(f \times g)(x_1, x_2) = (f(x_1), g(x_2))$, for all $f \times g \in \mathcal{F}^{(1)} \times \mathcal{F}^{(2)}$, and the metric on space  $X_1 \times X_2$ is given by $d((x_1, x_2),(y_1, y_2)) = \max\{d_1(x_1, y_1), d_2(x_2, y_2)\}$. Where $\varphi_{1}\times\varphi_{2}\in C\left( X_{1}\times X_{2},\mathbb{R}\right)$ is defined by $\left(\varphi _{1}\times\varphi _{2}\right)\left(x_{1},x_{2}\right)=\varphi_{1}\left( x_{1}\right)+\varphi _{2}\left( x_{2}\right)$.

\end{theorem}

\begin{proof}
Firstly, $\mathcal{F}^{(1)} \times \mathcal{F}^{(2)}$ is a set of finite continuous transformations on $X_1 \times X_2$. For any $\nu = p_n \ldots p_1\in F_{mk}^+$, there exist unique $w^{(1)} = j_n^{(1)} \ldots j_1^{(1)} \in F_m^+$ and unique $w^{(2)} = j_n^{(2)} \ldots j_1^{(2)} \in F_k^+$ such that $(f \times g)_{p_i} = f_{j_i^{(1)}}^{(1)} \times f_{j_i^{(2)}}^{(2)}$ for any $1\leq i \leq n$ and thus $(f \times g)_{\nu} = f_{w^{(1)}}^{(1)} \times f_{w^{(2)}}^{(2)}$.
On the other hand, if $w^{(1)} = j_n^{(1)} \ldots j_1^{(1)} \in F_m^+, w^{(2)} = j_n^{(2)} \ldots j_1^{(2)} \in F_k^+$, there exists unique $\nu = p_n \ldots p_1 \in F_{mk}^+$ such that $f_{j_i^{(1)}}^{(1)} \times f_{j_i^{(2)}}^{(2)}=(f \times g)_{p_i}$ for any $1\leq i \leq n$ and thus $f_{w^{(1)}}^{(1)} \times f_{w^{(2)}}^{(2)}=(f \times g)_{\nu}.$
 Thus for any $n \geq 1$, the map $ h:\nu \mapsto (w^{(1)}, w^{(2)})$ is a one-to-one correspondence.

 For $\varepsilon>0$ and $\nu \in F_{mk}^+$, there exist $w^{(1)}\in F_{m}^+$ and $w^{(2)}\in F_{k}^+$ such that $(f \times g)_{\nu} = f_{w^{(1)}}^{(1)} \times f_{w^{(2)}}^{(2)}$.

 If $E_1$ is a $(w^{(1)}, \varepsilon, r, \mathcal{F}^{(1)})$-separated set of $X_1$ and for any $x_1, y_1\in E_1$, we have
\[
 \frac{1}{|w^{(1)}|} Card\{w:d_1(f^{(1)}_{w}x_1,f^{(1)}_{w}y_1)\geq \varepsilon,~and~w\leq w^{(1)}\}> r.
 \]

Let\[
A= \{w:d_1(f^{(1)}_{w}x_1,f^{(1)}_{w}y_1)\geq \varepsilon,~and~w\leq w^{(1)}\},
\]
then we have $CardA> |w^{(1)}|r.$

Similarly, if $E_2$ is a $(w^{(2)}, \varepsilon, r, \mathcal{F}^{(2)})$-separated set of $X_2$ and then for any $x_2, y_2\in E_2$, we have
 \[
 \frac{1}{|w^{(2)}|} Card\{w:d_2(f^{(2)}_{w}x_2,f^{(2)}_{w}y_2)\geq  \varepsilon,~and~w\leq w^{(2)}\}> r.
 \]

Let
\[
B= \{w:d_2(f^{(2)}_{w}x_2,f^{(2)}_{w}y_2)\geq \varepsilon,~and~w\leq w^{(2)}\},
\]
then we have $CardB> |w^{(2)}|r.$

So for any $\nu' = p_k \cdots p_1\leq h^{-1}(w^{(1)},w^{(2)})$ and $1 \leq k \leq n$, we have
\begin{align*}
& d \Big((f \times g)_{\nu'}(x_1, x_2), (f \times g)_{\nu'}(y_1, y_2)\Big)\\
= & d\Big((f_{j_k^{(1)}\cdots j_1^{(1)}}^{(1)} (x_1), f_{j_k^{(2)}\cdots j_1^{(2)}}^{(2)} (x_2)), (f_{j_k^{(1)}\cdots j_1^{(1)}}^{(1)} (y_1), f_{j_k^{(2)}\cdots j_1^{(2)}}^{(2)} (y_2))\Big)\\
= & \max \left\{d_1\Big(f_{j_k^{(1)}\cdots j_1^{(1)}}^{(1)} (x_1), f_{j_k^{(1)}\cdots j_1^{(1)}}^{(1)} (y_1)\Big),~ d_2\Big(f_{j_k^{(2)}\cdots j_1^{(2)}}^{(2)} (x_2), f_{j_k^{(2)}\cdots j_1^{(2)}}^{(2)} (y_2)\Big)\right\}.
\end{align*}

Let $C= \left\{\nu': d\Big((f \times g)_{\nu'}(x_1, x_2), (f \times g)_{\nu'}(y_1, y_2)\Big)\geq \varepsilon
,~and~\nu'\leq \nu \right\}$, where $\nu=h^{-1}(w^{(1)},w^{(2)})$ , then
\[
Card C\geq CardA >|w^{(1)}|r
\]
and
\[
Card C\geq CardB >|w^{(2)}|r.
\]

Since $|w^{(1)}|=|w^{(2)}|=|\nu|$, then
\[
 \frac{1}{|\nu|}Card \left\{\nu': d\Big((f \times g)_{\nu'}(x_1, x_2), (f \times g)_{\nu'}(y_1, y_2)\Big)\geq \varepsilon ,~and~\nu'\leq \nu \right\}>r.
\]

Therefore, $E_1 \times E_2$ is a $(\nu, \varepsilon, r , \mathcal{F}^{(1)} \times \mathcal{F}^{(2)})$-separated set of $X_1 \times X_2$. Since

\begin{align*}
& \sum_{\left(x_{1}, x_{2}\right) \in E_{1} \times E_{2}} \exp \left(\sum_{\nu'\leq \nu}\left(\varphi_{1} \times \varphi_{2}\right)(f \times g)_{\nu'}\left(x_{1}, x_{2}\right)\right) \\
= & \left(\sum_{x_{1} \in E_{1}} \exp \left(\sum_{w^{\prime} \leq w^{(1)}} \varphi_{1}(f)_{w_{1}^{\prime}}\left(x_{1}\right)\right)\right)\left(\sum_{x_{2} \in E_{2}}\left(\sum_{w^{\prime} \leq w^{(2)}} \varphi_{2}(g)_{w_{2}^{\prime}}\left(x_{2}\right)\right)\right),
\end{align*}
then we obtain
\[
Q_{{\nu}}^{s}\left(\mathcal{F}^{(1)} \times \mathcal{F}^{(2)}, \varphi_{1} \times \varphi_{2}, \varepsilon, r\right) \geq Q_{{w}^{(1)}}^{s}\left(\mathcal{F}^{(1)}, \varphi_{1}, \varepsilon, r\right)  \cdot Q_{{w}^{(2)}}^{s}\left(\mathcal{F}^{(2)}, \varphi_{2}, \varepsilon, r\right).
\]

Then
\begin{align*}
& \frac{1}{(m k)^{n}} \sum_{|w|=n} Q_{w}^{s}\left(F^{(1)} \times F^{(2)}, \varphi_{1} \times \varphi_{2}, \varepsilon, r\right) \\
\geq & \frac{1}{(m k)^{n}} \sum_{\left|w^{(1)}\right|=n,\left|x^{(2)}\right|=n} Q_{w^{(1)}}^{s}\left(F^{(1)}, \varphi_{1}, \varepsilon, r\right) \cdot Q_{w^{(2)}}^{s}\left(F^{(2)}, \varphi_{2}, \varepsilon, r\right) \\
= & \frac{1}{m^{n}} \sum_{\left|w^{(1)}\right|=n} Q_{w^{(1)}}^{s}\left(F^{(1)}, \varphi_{1}, \varepsilon, r\right) \cdot \frac{1}{k^{n}} \sum_{\left|w^{(2)}\right|=n} Q_{w^{(2)}}^{s}\left(F^{(2)}, \varphi_{2}, \varepsilon, r\right).
\end{align*}

Thus
\[
Q_{n}^{s}\left(\mathcal{F}^{(1)} \times \mathcal{F}^{(2)}, \varphi_{1} \times \varphi_{2}, \varepsilon, r\right) \geq Q_{n}^{s}\left(\mathcal{F}^{(1)}, \varphi_{1}, \varepsilon, r\right) \cdot Q_{n}^{s}\left(\mathcal{F}^{(2)}, \varphi_{2}, \varepsilon, r\right).
\]

Moreover,
\begin{align*}
& \limsup_{n \to \infty} \frac{1}{n} \log Q_{n}^{s}\left(\mathcal{F}^{(1)} \times \mathcal{F}^{(2)}, \varphi_{1} \times \varphi_{2}, \varepsilon, r\right)\\
\geq & \liminf_{n \to \infty} \frac{1}{n} \log Q_{n}^{s}\left(\mathcal{F}^{(1)}, \varphi_{1}, \varepsilon, r\right) +
\limsup_{n \to \infty} \frac{1}{n} \log Q_{n}^{s}\left(\mathcal{F}^{(2)}, \varphi_{2}, \varepsilon, r\right).
\end{align*}

Since
$P^s_{{r,d}_1}(\mathcal{F}^{(1)},\varphi_{1})=\lim \limits_{\varepsilon\to 0}\liminf\limits_{n\to \infty}\frac{1}{n} \log Q_{n}^{s}\left(\mathcal{F}^{(1)}, \varphi_{1}, \varepsilon, r\right)$, and letting $\varepsilon \rightarrow 0$, we obtain
\[
P^s_{r,d}(\mathcal{F}^{(1)} \times \mathcal{F}^{(2)}, \varphi_{1} \times \varphi_{2}) \geq P^s_{r,d_1}(\mathcal{F}^{(1)},\varphi_{1})+ P^s_{r,d_2}(\mathcal{F}^{(2)},\varphi_{2}).
\]

The same reasoning proves another case.
\end{proof}

So far, we don't know if we can get any other inequality, that is, $P_{r,d}(\mathcal{F}^{(1)} \times \mathcal{F}^{(2)}, \varphi_{1} \times \varphi_{2}) \leq P_{r,d_1}(\mathcal{F}^{(1)},\varphi_{1})+ P_{r,d_2}(\mathcal{F}^{(2)},\varphi_{2})?$


\section{The proof of  Theorem \ref{11111}}
In this section, we give the proof of Theorem \ref{11111}. The  theorem  connects  the topological $r$-pressure  and the topological pressure of  free semigroup actions.

Before we proof Theorem \ref{11111}, using the similar method of Bufetov\cite{Bufetov} and Lin et al\cite{Lin} we obtain the following lemma.

\begin{lemma}\label{lem:skewright}
For any $n \geq 1$ , $\varepsilon>0$, $\varphi\in C\left( X,\mathbb{R} \right)$ and $g(\omega,x)=\varphi(x)$, we have
\begin{equation}\label{eq:rightSkew}
Q_{n}(F, g,\varepsilon, r) \leq K(\varepsilon,r)m^{n}Q_{n}(f_0, \cdots, f_{m-1}, \varphi,\varepsilon, r ),
\end{equation}
where $F$ is a skew product transformation and  $K(\varepsilon,r)$ is a positive constant that depends  on $\varepsilon$ and $r$.
\end{lemma}

\begin{proof}

Let $C(\varepsilon)$ be a positive integer satisfying $2^{-C(\varepsilon)} < \frac{\varepsilon}{100}$ and $N = m^{n + 2C(\varepsilon)}$. There are $N$ distinct words of length $n + 2C(\varepsilon)$ in $F_m^+$. Denote these words by $w_1, \cdots, w_{N}$. For each $1\leq i \leq N$, choose $\omega(i)\in \Sigma_m$ such that $\omega{(i)}|_{[-C(\varepsilon), n+C(\varepsilon)-1]}= w_i$. Obviously for $0\leq \varepsilon \leq1/2$, the sequences $\omega(i), i=1, \ldots, N$ form an $\left(n, \varepsilon, r, \sigma_{m}\right)$-spanning subset of $\Sigma_m$. Denote$\quad w_{i}^{\prime}=\left.\omega(i)\right|_{[0, n-1]}$, let $B_{i}$ denote the smallest cardinality of $\left(\bar{w}_{i}^{\prime}, \varepsilon, r, f_{0}, \ldots, f_{m-1}\right)$-spanning subset of $X$, $i=1,2,\ldots,N.$ And assume that the points $x_{1}^{\prime}, \ldots, x_{B_{i}}^{\prime}$ form a $\left(\bar{w}_{i}^{\prime}, \varepsilon, r, f_{0}, \ldots, f_{m-1}\right)$-spanning subset of $X$. Then the points
\[
(\omega(i), x_{j}^{i}) \in \Sigma_m \times X, \quad i=1, \ldots, N, j=1, \ldots, B_{i},
\]
form an $(n, \varepsilon, r, F)$-spanning subset of $\Sigma_m \times X .$ Hence, we have
\[
\begin{aligned}
Q_{n}(F, g, \varepsilon, r) & \leq \sum_{(\omega, x) \in\left\{\left(\omega(i), x_{j}^{i}\right):i=1,\ldots,N,~j=1,\ldots,B_i\right\}} e^{S_{n} g(\omega, x)} \\
& \leq K(\varepsilon, r) \sum_{\left|\bar{w}_{i}^{\prime}\right|=n, x \in\left\{x_{j}^{\prime}, j=1, \ldots, B_{i}\right\}} e^{S_{\bar{w}_{i}} \varphi(x)},
\end{aligned}
\]
where $K(\varepsilon, r)$ is a positive constant that depends on $\varepsilon$ and $r$.
Hence,
\[
Q_{n}(F, g, \varepsilon, r) \leq K(\varepsilon, r)m^{n} Q_{n}\left(f_{0}, \ldots, f_{m-1}, \varphi, \varepsilon, r\right).
\]

\end{proof}


\begin{proof}[The proof of Theorem \ref{11111}]
For $0<r<1,$ $\varphi\in C\left( X,\mathbb{R} \right)$, by Remark \ref{1234}, we have
\[
P_{r,d}(f_0, \ldots, f_{m-1},\varphi)\leq P_d(f_0, \ldots, f_{m-1},\varphi),
\]
then
\[
\limsup_{r \to 0}P_{r,d}(f_0, \ldots, f_{m-1},\varphi)\leq P_d(f_0, \ldots, f_{m-1},\varphi).
\]

From Lemma \ref{lem:skewright} we have
\[
Q_{n}(F,g,\varepsilon, r) \leq K(\varepsilon,r)m^nQ_{n}( f_0, \ldots, f_{m-1},\varphi,\varepsilon, r),
\]
whence
\[
P_{r,D}(F,g) \leq\log m +P_{r,d}(f_0, \ldots, f_{m-1},\varphi).
\]

From  Theorem \ref{222} and Theorem \ref{them:topology2}, we have
\begin{align*}
\liminf_{r \to 0}P_{r,d}(f_0, \ldots, f_{m-1},\varphi)&\geq \liminf_{r\to 0}P_{r,D}(F,g)-\log m\\
&=P_D(F,g)-\log m\\
&= P_d(f_0, \ldots, f_{m-1},\varphi).
\end{align*}

Thus
\[
  \limsup_{r \to 0}P_{r,d}(f_0, \ldots, f_{m-1},\varphi) \leq P_d(f_0, \ldots, f_{m-1},\varphi) \leq \liminf_{r \to 0}P_{r,d}(f_0, \ldots, f_{m-1},\varphi).
\]

Hence
\[
P_d(f_0, \ldots, f_{m-1},\varphi)=\lim_{r\to 0}P_{r,d}(f_0, \ldots, f_{m-1},\varphi).
\]
\end{proof}

\begin{corollary}\label{2}
Let $(X,d)$ be a compact metric space and  $f_0, f_1, \ldots, f_{m-1}$ are  continuous maps from $X$ into itself. Then
 \[
\lim_{r\to 0}P^s_r(f_0, \ldots, f_{m-1},\varphi)= P(f_0, \ldots, f_{m-1},\varphi).
 \]
\end{corollary}

\begin{proof}
From Theorem \ref{11}, we have
\[
P^s_{2r}(f_0, \ldots, f_{m-1},\varphi)-2Mr\leq P_r(f_0, \ldots, f_{m-1},\varphi)\leq P^s_r(f_0, \ldots, f_{m-1},\varphi),
\]
then
\begin{align*}
\limsup_{r \to 0}P^s_{2r}(f_0, \ldots, f_{m-1},\varphi)&\leq \limsup_{r\to 0}P_{r}(f_0, \ldots, f_{m-1},\varphi),\\
\liminf_{r\to 0}P_{r}(f_0, \ldots, f_{m-1},\varphi)&\leq \liminf_{r \to 0}P^s_{r}(f_0, \ldots, f_{m-1},\varphi).
\end{align*}

By Theorem \ref{11111},
\[
\lim_{r\to 0}P_r(f_0, \ldots, f_{m-1},\varphi)= P(f_0, \ldots, f_{m-1},\varphi),
\]
then
\[
\limsup_{r \to 0}P^s_{2r}(f_0, \ldots, f_{m-1},\varphi)\leq P(f_0, \ldots, f_{m-1},\varphi)
\leq \liminf_{r \to 0}P^s_{r}(f_0, \ldots, f_{m-1},\varphi).
\]

This implies the existence of the limit, therefore
\[
\lim_{r \to 0}P^s_{r}(f_0, \ldots, f_{m-1},\varphi)= P(f_0, \ldots, f_{m-1},\varphi).
\]
\end{proof}

\section{The proof of  Theorem \ref{0000} }

Before we prove  Theorem \ref{0000}, we should make some preparations.

Let $(X,d)$ be a compact metric space, suppose a free semigroup with $m$ generators acts on $X$ and the generators $f_0, f_1, \ldots, f_{m-1}$ of $X$ are homeomorphisms. It's easy to see that the skew-product transformation
$F: \Sigma_m \times X \to \Sigma_m \times X$ is homeomorphism.

We can see that the  skew-product inverse transformation $F^{-1}: \Sigma_m \times X \to \Sigma_m \times X$ and $g:\Sigma_m \times X \to \mathbb{R}$ are defined by the formula
\[
F^{-1}(\omega, x) = (\sigma^{-1} \omega, f^{-1}_{\omega_{-1}}(x)),
\]
\[
g(\omega, x)=\varphi(x),
\]
where $\varphi\in C\left( X,\mathbb{R} \right)$ and $\omega=(\ldots,\omega_{-1},\overset{*}\omega_0,\omega_1,\cdots)\in \Sigma_m$. Here $f_{\omega_{-1}}$ stands for $f_0$ if $\omega_{-1}=0$, and for $f_1$ if $\omega_{-1}=1$, and so on. Moreover
\begin{align*}
F^{-n}(\omega,x)&=(\sigma^{-n}\omega,f^{-1}_{\omega_{-n}}f^{-1}_{\omega_{-n+1}}\cdots f^{-1}_{\omega_{-1}}(x))\\&=(\sigma^{-n}\omega,f^{-1}_{\overline{\omega|_{[-n,-1]}}}(x)).
\end{align*}

In order to prove   Theorem \ref{0000}, we give the following property of $P(F^{-1},g).$

\begin{theorem}\label{33}
The topological pressure of the skew product inverse transformation $F^{-1}$ satisfies
\begin{equation}
P_D(F^{-1},g) =\log m + P_d(f^{-1}_0, \ldots, f^{-1}_{m-1},\varphi),
\end{equation}
where the metric $D$ on $\Sigma_m \times X$ is defined as
\[
D((\omega,x), (\omega', x^\prime)) = \max\{d'(\omega, \omega'), d(x, x')\}
\]
and the metric $d'$ on $\Sigma_m$ is introduced by setting $d'(\omega,\omega')=1/2^k$, and $k=\inf\{|n|:\omega_n\neq\omega_n^{'} \}$.
\end{theorem}

The proof of the theorem is similar to that of Lin et al\cite{Lin}, so we omit the proofs.

\begin{lemma}\label{lem:skewleft2}
For any  $n \geq 1$ and $0 \leq \varepsilon \leq \frac{1}{2}$, then
\begin{equation}\label{eq:leftSkew2}
Q_{n}^s(F^{-1},g, \varepsilon ) \geq m^{n} Q_{n}^s(f^{-1}_0, \ldots, f^{-1}_{m-1},\varphi, \varepsilon).
\end{equation}
\end{lemma}

\begin{lemma}\label{lem:skewright2}
For any $n \geq 1$ and $\varepsilon >0$, then
\begin{equation}\label{eq:rightSkew2}
Q_{n}(F^{-1},g,\varepsilon) \leq K(\varepsilon)m^{n} Q_{n}(f^{-1}_0, \ldots, f^{-1}_{m-1},\varphi, \varepsilon ),
\end{equation}
where $K(\varepsilon)$ is a positive constant that depends only on $\varepsilon$.
\end{lemma}

\begin{proof}[The proof of Theorem \ref{33}]
From Lemma \ref{lem:skewleft2}, we have
\[
P_D(F^{-1},g) \geq\log m + P_d( f^{-1}_0, \ldots, f^{-1}_{m-1},\varphi).
\]
From Lemma \ref{lem:skewright2}, we have

\[
P_D(F^{-1},g) \leq\log m + P_d( f^{-1}_0, \ldots, f^{-1}_{m-1},\varphi),
\]
and the proof is complete.
\end{proof}
Now we give the proof of Theorem \ref{0000}.

\begin{proof}[The proof of Theorem \ref{0000}]
Since $F$ is homeomorphism, we have
\[
P_D(F^{-1},g)=P_D(F,g).
\]

From  Theorem \ref{33} and Theorem \ref{them:topology2}, we have
\begin{align*}
P_D(F^{-1},g) &=\log m + P_d( f^{-1}_0, \ldots, f^{-1}_{m-1},\varphi),\\
P_D(F,g) &= \log m + P_d( f_0, \ldots, f_{m-1},\varphi).
\end{align*}

Therefore
\[
P_d( f_0, \ldots, f_{m-1},\varphi) =P_d( f^{-1}_0, \ldots, f^{-1}_{m-1},\varphi).
\]
\end{proof}

\begin{remark}
In \cite{Ma}, the authors gave an example to show that topological pressure in the sense of article \cite{Ma} does not have the property similar to Theorem 1.5.
\end{remark}

\bibliographystyle{amsplain}

\end{document}